\documentclass[11pt]{amsart}

\usepackage[T1]{fontenc}
\usepackage{amsfonts,amssymb,amsthm, txfonts, pxfonts,amscd}
\usepackage[stretch=10]{microtype}
\usepackage{epsfig}
\usepackage{mathrsfs}
\usepackage{dsfont}

\newtheorem{thm}{Theorem}
\newtheorem{prop}[thm]{Proposition}
\newtheorem{lem}[thm]{Lemma}
\newtheorem{cor}[thm]{Corollary}

\theoremstyle{definition}

\theoremstyle{remark}

\newcommand{\R}{\mathbb{R}}

\newcommand{\zz}[1]{\mathbb #1}

\newcommand{\dd}{\mathrm{d}}

\title[Critical Branching Symmetric Stable
Processes]{Maximal Displacement of Critical Branching Symmetric Stable
Processes}
\author{Steven P. Lalley}
\address{Department of Statistics, University of Chicago, Chicago, IL 60637}
\email{lalley@galton.uchicago.edu}
\urladdr{www.statistics.uchicago.edu/$\sim$lalley}
\author{Yuan Shao}
\address{Department of Mathematics, University of Chicago, Chicago, IL 60637}
\email{shaoyuan3319@gmail.com}
\date{\today}
\thanks{First author supported by NSF grant DMS  -
1106669}

\begin{document}

\begin{abstract}
We consider a critical continuous-time branching process (a Yule
process) in which the individuals independently execute symmetric
$\alpha-$stable random motions on the real line starting at their
birth points. Because the branching process is critical, it will
eventually die out, and so there is a well-defined maximal location
$M$ ever visited by an individual particle of the  process. We prove
that the distribution of $M$ satisfies the asymptotic relation
$P\{M\geq x \}\sim (2/\alpha)^{1/2}x^{-\alpha /2}$ as $x \rightarrow \infty$.
\end{abstract}

\maketitle

\section{Introduction and Main Result}

\subsection{Critical Branching Symmetric Stable Process}

The subject of  this paper is the \emph{critical
branching symmetric $\alpha-$stable process} (henceforth abbreviated
as ``CBSS process''), a critical branching process
in which each particle also moves in a one-dimensional space
according to a symmetric stable process. Formally, this is defined to be a
continuous-time stochastic particle system initiated by a single particle at
the origin $0 \in \R$  whose reproduction and dispersal mechanisms are
as follows:

(A) Each particle, independently of all others and of the past of the
process, waits an exponentially distributed time with parameter $1$,
and then either splits into two identical particles or dies with
probability $1/2$;

(B) While not branching, each particle moves in $\R$ following a
symmetric stable process of exponent $\alpha$,  independent of the
reproduction process.

Since the branching mechanism is  critical, the process will go
extinct in finite time, with probability one. Consequently, there is a
unique maximal real number $M\geq 0$, which we dub the  \emph{maximal
displacement} of the process, such that some particle of the process
reaches the location $M$. Our interest is in the tail of the
distribution of $M$. Let $Y_{t}$ ($Y$ for \emph{Yule}, as this is a
Yule process) be the total number of particles at
time $t$, and let $\zeta _{t,i}$ $(i=1,2,\cdots,Y_{t})$ be the locations of
the particles. Then the maximal displacement random variable is
formally defined by
\begin{equation*}
M=\max_{t \geq 0} M_t \quad \text{where} \quad M_t=\max_{i=1,2,\cdots,Y_{t}} \zeta _{t,i},
\end{equation*}
with the convention that the max of the empty set is $-\infty$.
The main result of the paper is the following theorem.

\begin{thm}\label{mainresult}
Let $M$ be the maximal displacement of a critical branching symmetric
stable process of exponent $\alpha$ ($0<\alpha<2$). Then
\begin{equation}\label{mainresultequation}
P \{ M \geq x \} \sim \sqrt{\frac{2}{\alpha}} \frac{1}{x^{\alpha /2}}
\quad \text{as $x \rightarrow \infty$.} 
\end{equation}
\end{thm}

\subsection{Discussion}\label{ssec:brw} Theorem~\ref{mainresult} is
the natural analogue for the CBSS process of a theorem describing the
maximal displacement of a critical driftless branching random walk,
recently proved by the authors in \cite{lalley-shao:1}. The critical
branching random walk is a discrete-time branching process in which
particles alternately reproduce and move, as follows. The reproduction
steps are governed by the law of a critical Galton-Watson process
whose offspring distribution has finite variance $\eta^{2}$ and finite
third moment; the movement steps are governed by the law of a
finite-variance, mean-zero random walk on the integers $\zz{Z}$. See
\cite{lalley-shao:1} or \cite{kesten} for further details on the
construction of the process. Since the branching mechanism is assumed
to be critical (that is, the offspring distribution has mean $1$), the
process dies out after finitely many generations, and hence there is a
well-defined maximal displacement random variable $M$, defined in the
same manner as for the CBSS process discussed above. The main result
of \cite{lalley-shao:1} states that if the step distribution of the
random walk component of the branching random walk has mean $0$,
variance $\sigma^{2}>0$, and finite $4+\varepsilon$ moment for some
$\varepsilon >0$, then as $x \rightarrow \infty$,
\begin{equation}\label{eq:brwAsymptotics}
	P\{M\geq x \}\sim \frac{6\eta^{2}}{\sigma^{2}x^{2}}.
\end{equation}

The result \eqref{eq:brwAsymptotics} is itself the natural extension
to branching random walks of an earlier result of Sawyer \& Fleischman
\cite{fleischman-sawyer} for critical branching Brownian
motion.\footnote{See also \cite{lee}. The paper of Sawyer and
Fleischman proposed the critical branching Brownian motion as a model
for the dispersal of a mutant but neutral allele in a homogeneous
environment. Branching random walks and branching diffusion processes
are also used as models in combustion and reaction-diffusion
processes, and they occur as low-density limits of certain spatial
epidemics \cite{lalley:spatial}. In all of these situations the
maximal displacement is of natural interest.} For branching Brownian
motion it is possible (and not difficult) to write a second-order
ordinary differential equation for the distribution function of $M$;
the tail asymptotics of solutions can then be obtained by relatively
standard methods in ODE theory. See \cite{fleischman-sawyer} for
details. For branching random walk, it is also quite easy to write a
nonlinear convolution equation for the distribution function
(cf. \cite{lalley-shao:1}, \cite{kesten}), but ODE methods cannot be
used to determine tail asymptotics. (See the discussion on p.~924 of
\cite{kesten}, in particular his eqn. 1.9, where the tail asymptotics
is left as an open problem.) The primary technical contribution of
\cite{lalley-shao:1} is a new method, based on Feynman-Kac formulas, for
the analysis of such nonlinear convolution equations.  The bulk of
this paper will be devoted to a parallel method for studying the
asymptotic behavior of solutions to certain \emph{pseudo-differential}
equations that will be shown to govern the distribution function of
the maximal displacement random variable for the CBSS process.

\subsection{Heuristics}\label{ssec:heuristics} The following heuristic
arguments suggest that $x^{-\alpha /2}$ is the correct order of
magnitude for the tail probability \eqref{mainresult}. Fix a large
time $T$ and consider the event that the branching process survives to
time $T$: by Kolmogorov's theorem for critical branching processes
(cf. \cite{athreya-ney}, ch.~1), the chance of this is on the order of
$1/T$. Furthermore, by Yaglom's theorem, if the branching process
survives to time $T$ then at typical times $t\in [\varepsilon T,T]$
the number of particles alive will be on the order of $T$. Thus, the
total ``particle-time'' will be on the order of $T^{2}$. Now in each
small interval $(\Delta t)$ of time, each particle has a chance
$(\Delta t) T^{-2}$ of jumping a distance more than $T^{2/\alpha}$ to
the right (by the Poisson point process representation of the
symmetric $\alpha-$stable process: see section~\ref{sec:preliminaries}
below). Since the total particle-time is on the order of $T^{2}$, it
follows that the conditional probability that some particle makes it
past location $T^{2/\alpha}$ is on the order $O (1)$. Thus, the
\emph{unconditional} probability is on the order of $1/T$, since this
is the probability that the process survives to time $T$. A similar
argument shows that unless the process survives for significantly
longer than time $T$ then the chance that a particle moves much
farther right than $T^{2/\alpha}$ is negligible.

These heuristics can, with some care (see
section~\ref{sectionaprioribound} below), be made into rigorous
arguments to prove that
\begin{equation}\label{eq:orderOfMagnitude}
	P\{M\geq x \}\asymp x^{-\alpha /2},
\end{equation}
but there is little hope of obtaining the sharp asymptotic formula
\eqref{mainresult} by similar methods. (See \cite{kesten} for a
detailed analysis of such arguments in the case of a critical,
driftless branching random walk.) The rough asymptotic formula
\eqref{eq:orderOfMagnitude} will be a necessary preliminary step in
proving the sharper result \eqref{mainresult} (see
Proposition~\ref{aprioribound} in
section~\ref{sectionaprioribound}), but \eqref{mainresult} will also
require the use of different tools based on Feynman-Kac formulas.

\subsection{Superprocess limits}\label{ssec:superprocessLims} The
asymptotic relation \eqref{eq:brwAsymptotics} is shown in
\cite{lalley-shao:1} to be closely related to the Dawson-Watanabe
scaling limit (super-Brownian motion) for critical branching random
walks. There is a similar relation between the asymptotic formula
\eqref{mainresult} for the CBSS process and the superprocess for the
symmetric $\alpha-$stable process. (See, e.g., \cite{legall}, ch.~2
for an introduction to the basic theory of these superprocesses.) In
brief, if $n$ independent copies of the CBSS are all started at time
$0$ at the origin, if time and mass are scaled by $n$ and space is
scaled by $n^{\alpha /2}$ then the resulting measure-valued process
becomes the symmetric $\alpha-$stable superprocess in the $n
\rightarrow \infty$ limit. This scaling is consistent with
\eqref{mainresult}, because by Kolmogorov's theorem, among the $n$
branching processes the number that survive to time $n$ is
(approximately) Poisson with mean $1$, and so \eqref{mainresult}
suggests that $n^{\alpha /2}$ is the right scaling of space for the
superprocess limit. Of course, \eqref{mainresult} cannot be deduced
from the existence of the superprocess limit, because the maximum
location visited might be determined by a small $o (n)$ number of
particles that drift away from the bulk of the mass. In fact, the
result \eqref{mainresult} can be interpreted as asserting that this
does not happen. See \cite{lalley-shao:1} for an extended discussion of
the analogous point for the finite variance case.

\subsection{Plan of the paper}\label{ssec:plan}
Theorem~\ref{mainresult} will be proved by first showing that the
distribution function of $M$ satisfies a pseudo-differential equation
\eqref{initialvalueproblem} involving the fractional Laplacian
operator. This will be done in section~\ref{section:pseudo-DE}. A
comparison principle for solutions to the pseudo-differential equation
will be proved in section~\ref{sectionaprioribound}, and this will be
used to prove the \emph{a priori} estimates
\eqref{eq:orderOfMagnitude}. Finally, in
section~\ref{sectionfeynmankac}, a Feynman-Kac representation of
solutions to the pseudo-differential equation
\eqref{initialvalueproblem} will be used to obtain sharp asymptotics.
The Feynman-Kac representation will involve path integrals of the
symmetric $\alpha-$stable process, and in analyzing these it will be
necessary to call on some structural features of these processes: the
relevant facts are collected in section~\ref{sec:preliminaries}.

\section{Preliminaries on Symmetric Stable Processes}
\label{sec:preliminaries}

Recall \cite{bertoin} that a \textit{symmetric $\alpha -$stable
process} in $\R$ is a real-valued L\'evy process $\{ X_t \} _{t \geq
0}$ whose distribution $X_t$ is symmetric (i.e. $X_t$ has the same
distribution as $-X_t$) for any $t \geq 0$, and satisfies the scaling
property
\begin{equation}\label{scalingproperty}
\frac{X_t}{t^{1/\alpha}} \overset{\mathscr{D}}{=} X_1 \qquad \forall t>0.
\end{equation}
Henceforth, we shall reserve the symbol $X$ for a symmetric
$\alpha-$stable process, and we shall use the usual convention of
attaching a superscript $x$ to the probability and expectation
operators $E^{x},P^{x}$ to denote that under $P^{x}$ the 
process $X_{t}$ has initial value $X_{0}=x$. When the superscript $x$
is omitted, it should be understood that $x=0$.

The characteristic function of a symmetric $\alpha-$stable process has
the form $Ee^{i \theta X(t)} = \exp(-\gamma t |\theta|^{\alpha})$ for
some constant $\gamma > 0$.  This is clearly integrable, and so it
follows by the Fourier inversion theorem that the distribution of
$X_{t}$ has a density $f_{t} (x)$ with respect to Lebesgue measure
$dx$. We shall assume time is scaled so that $\gamma =1$, and we shall
only consider the case $\alpha <2$. For such $\alpha$, the symmetric
stable process is a pure jump process and has a Poisson point process
representation
\begin{equation}\label{pointprocessrepresentation}
X(t) 
 = \iint_{\mathscr{X}} y \mathds{1}_{(0,t]}(s) \, N(\dd s,\dd y),
\end{equation}
where $N (ds, dy)$ is a Poisson random measure with intensity $\mu$
given by
\begin{align}\label{intensitymeasure}
\mu(\dd t,\dd y) &= \dd t \cdot \lambda(\dd y)  \quad \text{with L\'{e}vy measure} \\
\notag 
\lambda(\dd y) &= |y|^{-1-\alpha} \dd y \quad \text{on $\R$}.
\end{align}

The infinitesimal generator of a symmetric $\alpha-$stable  process is
(see \cite{bertoin}, page 24) the \textit{fractional Laplacian
pseudo-differential operator} $-(-\Delta)^{\alpha/2}$. Thus, if $X_t$
is symmetric $\alpha-$stable and $f: \R \rightarrow \R$ is a
suitable function (for example, a compactly supported smooth function),
then
\begin{equation}\label{infinitesimalgenerator}
\lim_{t \rightarrow 0+} \frac{E^x[f(X_t)]-f(x)}{t} = -(-\Delta)^{\alpha/2} f(x),
\end{equation}
where $(-\Delta)^{\alpha/2}$ is the non-local linear operator defined by
the singular integral
\begin{equation}\label{fractionallaplacian}
(-\Delta)^{\alpha/2} f(x) := \int_{-\infty}^{\infty} \frac{f(x)-f(y)}{|x-y|^{1+\alpha}} \,\dd y.
\end{equation}
The domain of this operator is understood to be the set of all
bounded, continuous functions such that the limit
\eqref{infinitesimalgenerator} exists. 

Several elementary first-passage properties of the symmetric $\alpha
-$stable process $\{X_{t} \}_{t\geq 0}$ will be used repeatedly in the
analysis that follows. First, by the Hewitt-Savage 0--1 Law, the
limsup and liminf of any sample path are $\pm \infty$, and so for any
$A>0>B$ the first-passage times
\begin{align}\label{eq:fpt}
	\tau^{+}_{A}&=\tau^{+} (A)=\inf \{t>0 \,:\, X_{t}\geq A\} \quad \text{and}\\
\notag \tau^{-}_{B}&=\tau^{-} (B) =\inf \{t>0 \,:\, X_{t}\leq B\}
\end{align}
are almost surely finite. Note that by symmetry the random variables
$\tau^{+} (A)$ and $\tau^{-} (-A)$ have the same distribution, and
similarly so do $X_{\tau^{+} (A)}$ and $-X_{\tau^{-} (-A)}$.
Second, by the scaling law, for any $A>0$
the joint distribution of
$(\tau^{+} (A) /A^{\alpha},X_{\tau^{+} (A)}/A)$ is identical to that of
$(\tau^{+} (1),X_{\tau^{+} (1)})$. Third, there is the following
analogue of the ``reflection principle'' for Brownian motion.

\begin{lem}\label{reflectionprinciple}
Let $X_t$ be a symmetric $\alpha$-stable process with $X_0=0$, and
define 
\begin{equation}\label{eq:maxDef}
	 X_t^* := \max_{s \in [0,t]} X_s.
\end{equation}
Then for any $y>0$, $$P \{ X_t^* \geq y \} \leq 2 P \{ X_t \geq y \}.$$
\end{lem}

\begin{proof}
Fix $y>0$, and abbreviate $\tau =\tau_{y}$. Then
$X_\tau \geq y$, and by the strong Markov property,
\begin{equation*}
\tilde{X}_t :=
\begin{cases}
X_t \quad & \text{when $t<\tau$},\\
2X_\tau - X_t & \text{when $t \geq \tau$}
\end{cases}
\end{equation*}
is a symmetric $\alpha$-stable process as well. Hence
\begin{equation*}
\begin{aligned}
P \{ X_t^* \geq y \} & = P \{ X_t^* \geq y, X_t \geq y \} + P \{ X_t^* \geq y, X_t < y \}\\
& = P \{ X_t \geq y \} + P \{ \tau < t, \tilde{X}_t > 2X_\tau-y \}\\
& \leq P \{ X_t \geq y \} + P \{ \tilde{X}_t > y \}\\
& = 2 P \{ X_t \geq y \}.
\end{aligned}
\end{equation*}
\end{proof}

Some of the arguments in sections \ref{section:pseudo-DE} and
\ref{sectionfeynmankac} will require estimates on first-passage
probabilities on very short time intervals. For these, the following
asymptotic formulas will be useful.

\begin{lem}\label{lemma:shortTimeAsymptotics}
For any interval $J\subset \zz{R}$ let $\sigma_{J}$ be the first exit
time from the interval $J$.
For any fixed $0<\delta <A/2$, as $\varepsilon \rightarrow 0$,
\begin{gather}\label{eq:shortTimeAsymptotics}
	\lim_{\varepsilon \rightarrow 0}\varepsilon^{-1}P (\tau^{+}
	_{A}<\varepsilon)=  \lambda [A,\infty) 	\quad \text{and} 
	 \\
\label{eq:shortTimeAsymptoticsB}
	\lim_{\varepsilon \rightarrow 0} P (\tau^{+}_{A}\not
	 =\sigma_{(-\delta ,\delta)}\,|\, \tau^{+}_{A} <\varepsilon
	 )=0 .
\end{gather}
\end{lem}

\begin{proof}
By Lemma~\ref{reflectionprinciple} and the scaling law, for any
fixed $A>0$,
\begin{align*}
	P \{ \tau^{+} (A)<\varepsilon\}&\leq 2 P\{X_{\varepsilon}\geq A \}\\
	  &=P\{X_{1}\geq A\varepsilon^{-1/\alpha } \}\\
	  &\sim \kappa A^{\alpha}\varepsilon 
\end{align*}
where $\kappa >0$ is a constant that depends on the exponent
$\alpha$. (See, e.g., \cite{zolotarev}, p. 95, or \cite{feller2},
sec.~XVII.6 for the fact that the tail of the $\alpha -$stable law is
regularly varying with exponent $\alpha$.)  It follows by symmetry
that $P\{\sigma_{(-A,A)}<\varepsilon \}\sim 2\kappa
A^{\alpha}\varepsilon$. Consequently, for any
(small) $\delta >0$,
\[
	P (\tau^{+} (A)<\varepsilon \;\text{and}\;
	\tau^{+}_{A}\not =\sigma_{(-\delta ,\delta )}) =O (\varepsilon^{2}) 
\]
as $\varepsilon  \rightarrow 0$, because the event would require the
process to make two successive first passages of size $\delta$ before
time $\varepsilon$. Thus, the event $\tau^{+} (A)<\varepsilon$ is nearly
entirely accounted for by sample paths that make a single jump of size
$>A-2\delta$ before time $\varepsilon$; in particular, for any $\delta
>0$, as $\varepsilon  \rightarrow 0$
\begin{align*}
	P (\tau^{+} (A)<\varepsilon )&\geq P (N ([0,\varepsilon]\times
	[A+2\delta ,\infty])\geq 1)+O (\varepsilon^{2}) \quad \text{and}\\
	P (\tau^{+} (A)<\varepsilon )&\leq P (N ([0,\varepsilon]\times
	[A-2\delta ,\infty])\geq 1)+O (\varepsilon^{2}).
\end{align*}
Since $\delta >0$ can be chosen arbitrarily small, relations
\eqref{eq:shortTimeAsymptotics}--\eqref{eq:shortTimeAsymptoticsB}
follow. A similar argument proves \eqref{eq:shortTimeAsymptoticsC}.  
\end{proof}

Lemma~\ref{lemma:shortTimeAsymptotics} indicates that when the
first-passage time $\tau^{+} (A)$ is very small it is because the 
path makes a single jump of size $\geq A$ at time $\tau^{+} (A)$. The
following lemma -- a consequence of the Poisson point process
representation of the stable process -- asserts that the size of this
jump is independent of the path up to the jump time. 
For any interval $J\subset \zz{R}$, define $\nu _{J}=\nu  (J)$ to be the
first time that $X_{\nu  (J)}-X_{\nu  (J)-}\in J$, equivalently,
\begin{equation}\label{eq:firstJumpTime}
	\nu_{J}=\inf \{t>0 \,:\, N (t,J)=1\}.
\end{equation}

\begin{lem}\label{lemma:jumpIndependence}
Let $J\subset \zz{R}$ be a nonempty, open interval such that $0$ is in
the interior of $J^{c}$. Then the jump size $X_{\nu
(J)}-X_{\nu (J)-}$ is independent of $\mathcal{F}_{\nu (J)-}$.
\end{lem}

\begin{proof}
This is an elementary consequence of the Poisson point process
representation, using the fact that any event in the $\sigma -$algebra
$\mathcal{F}_{\nu (J)-}$ is determined by the restriction of the
Poisson point process to $[0,\infty )\times J^{c}$.
\end{proof}

\begin{cor}\label{corollary:jumpIndependence}
For all $x>A>0$, as $\varepsilon  \rightarrow 0$,
\begin{equation}\label{eq:shortTimeAsymptoticsC}
	\lim_{\varepsilon \rightarrow 0} P (X_{\tau^{+} _{A}}>x
	\,|\,\tau^{+} _{A}<\varepsilon) = \frac{\lambda  [x,\infty
	)}{\lambda  [A,\infty )} = (x/A)^{-\alpha}.
\end{equation}
Moreover, this relation holds uniformly in the region $x>A\geq  1$.
\end{cor}

\section{A Pseudo-Differential Equation for the CDF} \label{section:pseudo-DE}

Let $M$ be the maximal displacement of a CBSS process.  The tail
distribution function of $M$ will be denoted by
\begin{equation}\label{eq:tailDF}
	 u(x)=P \{ M \geq x \}.
\end{equation}
Clearly, $u (x)=1$ for all $x \leq 0$, and it is easily seen that $0<u
(x)<1$ for all $x>0$. Since $M<\infty$ with probability one, $\lim_{x
\rightarrow \infty} u (x)=0$. Furthermore, the strong Markov property
for the CBSS implies that $u$ is continuous, as the following argument
shows. Fix $x\geq 0$, and denote by $T_{x}$ the first time that a
particle of the CBSS process reaches $x$; this is a stopping time.  By
the strong Markov property, the post-$T_{x}$ process initiated by the
particle at $x$ is itself a CBSS process; this process will, with
(conditional) probability $1$, place a particle in $(x,\infty)$ at
some time after $T_{x}$, because (i) the initiating particle will not
immediately die, and (ii) a symmetric stable process started at $0$
must immediately enter both the positive and negative halflines.

The key to our analysis of the tail behavior of $u$ is that $u$
satisfies the following pseudo-differential equation.

\begin{prop}
$u(x)$ solves the following nonlinear boundary value problem
\begin{equation}\label{initialvalueproblem}
\begin{cases}
(-\Delta)^{\alpha/2} u(x) + \frac{1}{2}(u(x))^2 = 0 \quad & \text{for $x>0$},\\
u(x)=1 & \text{for $x \leq 0$}.
\end{cases}
\end{equation}
\end{prop}

\begin{proof}
Fix any $x>0$. Let us calculate $P \{ M<x \} = 1-u(x)$ by conditioning
on the evolution of the $CBSS$ process up to time $\varepsilon >0$. Up
until the time $T$ that it first fissions or dies, the initiating
particle follows a symmetric $\alpha$-stable trajectory. Let $\{X_s
\}_{s \geq 0}$ be a generic symmetric $\alpha$-stable process, and
write $X_t^* := \max_{s \in [0,t]} X_s$. By
Lemma~\ref{reflectionprinciple} (the ``reflection principle''), for
any fixed $x>0$, as $\varepsilon  \rightarrow 0$,
\begin{equation}\label{eq:short-time-max}
	P\{X^{*}_{\varepsilon}\geq x \}\leq 2P\{X_{\varepsilon}\geq x
	\}=O (\varepsilon). 
\end{equation}

The distribution of $T$ is  exponential with mean one, and the event
that the initiating particle fissions rather than dies at time $T$ is
Bernoulli-$1/2$, independent of $T$. If the initiating particle
fissions then the event $M<x$ requires that \emph{both} of the CBSS
processes engendered by the fission have maximal displacements $<x$
\emph{and} that the path of the initial particle up to the time of
fission stays below the level $x$. Hence,
\begin{align}\label{eq:term-1}
P \{ M<x, &\text{ initial particle fissions before time $\varepsilon$} \}\\
\notag 
= & \; \frac{1}{2} \int_0^\varepsilon e^{-t} \Bigl ( \int_{-\infty}^x \bigl ( P \{ M<x-y \} \bigr )^2 \,\dd F_t(y) - P \{ X_t^* \geq x \} \Bigr ) \,\dd t\\
\notag 
= & \; \frac{1}{2} \int_0^\varepsilon e^{-t}\int_{-\infty}^x \bigl ( 1-u(x-y) \bigr )^2 \,\dd F_t(y) \,dt + o(\varepsilon)\\
\notag 
= & \; \frac{1}{2} \, \varepsilon \, \bigl ( 1-u(x) \bigr )^2 + o(\varepsilon),
\end{align}
where $F_t$ is the distribution of the random variable $X_{t}$. The
last equality holds because $u$ is continuous and bounded and
$F_{t}\Rightarrow F_0=\delta_0$ as $t \rightarrow 0$.  The second
equality follows from the estimate \eqref{eq:short-time-max}. 
A similar argument shows that
\begin{align}\label{eq:term-2}
 P \{ M<x, &\text{ initial particle dies  before time $\varepsilon$} \}\\
\notag = & \; \frac{1}{2} \int_0^\varepsilon e^{-t} \cdot P^0 \{ X_t^* < x\} \,\dd t\\
\notag = & \; \frac{1}{2} \, \varepsilon + o(\varepsilon).
\end{align}

Next, by the Markov property,
\begin{align}\label{eq:term-3}
P \{ M<x \; & \text{and} \; T>\varepsilon \}\\
\notag = & \; e^{-\varepsilon} \Bigl ( \int_{-\infty}^x P \{ M<x-y \}
\,\dd F_\varepsilon(y) - P \{ X_\varepsilon^* \geq x, X_\varepsilon<x \} \Bigr
).  
\end{align}
The first term  in \eqref{eq:term-3}
is equal to
\begin{equation*}
\begin{aligned}
& \; e^{-\varepsilon} \int_{-\infty}^x \bigl ( 1-u(x-y) \bigr ) \,\dd F_\varepsilon(y)\\
= & \; e^{-\varepsilon} \Bigl ( \int_{-\infty}^{\infty} \bigl ( 1-u(x-y) \bigr ) \,\dd F_\varepsilon(y) - (1-u(x)) \Bigr ) + e^{-\varepsilon} (1-u(x))\\
= & \; e^{-\varepsilon} \int_{-\infty}^{\infty} \bigl ( u(x)-u(x-y) \bigr ) \,\dd F_\varepsilon(y) + e^{-\varepsilon} (1-u(x))\\
= & \; \varepsilon \, (-\Delta )^{\alpha /2} u (x) + (1-\varepsilon)(1-u(x)) + o(\varepsilon).
\end{aligned}
\end{equation*}
In the third equality we exploited the boundary condition $u(x-y)=1$
when $x-y \leq 0$. The last equality follows from
\eqref{infinitesimalgenerator}. 

The three terms \eqref{eq:term-1}, \eqref{eq:term-2}, and
\eqref{eq:term-3} account for all of the terms in the
pseudo-differential equation \eqref{initialvalueproblem}. Thus, to
complete the proof 
it remains only to show that the second term in \eqref{eq:term-3} satisfies
\[
	P \{ X_\varepsilon^* \geq x, X_\varepsilon<x \} = o(\varepsilon).
\]
For this we appeal to
Lemmas~\ref{reflectionprinciple}--\ref{lemma:shortTimeAsymptotics}.
Relation~\eqref{eq:shortTimeAsymptotics} of
Lemma~\ref{lemma:shortTimeAsymptotics} implies that $P (\tau
(x)\leq \varepsilon)=P (X^{*}_{\varepsilon})=O (\varepsilon)$, so it is
enough to show that as $\varepsilon  \rightarrow 0$,
\[
	P (X_{\varepsilon}<x \,|\, \tau (x)\leq \varepsilon)=o (1).
\]
Now if $X_{\tau (A)}>x+\beta$ then in order that $X_{\varepsilon}<x$
the process must traverse an interval of size $\beta$ in time
$<\varepsilon$, and by Lemma~\ref{reflectionprinciple} the chance of
this is no more than $2P (X_{\varepsilon}>\beta)$. Furthermore, by the
scaling law \eqref{scalingproperty}, if $\beta =\varepsilon^{\varrho}$
for some $\varrho <1/\alpha$ then $P (X_{\varepsilon}>\beta)=o (1)$ as
$\varepsilon  \rightarrow 0$. Consequently,
\begin{align*}
	P (X_{\varepsilon}<x \,|\,& \tau (x)\leq \varepsilon)\\
	&=	P (X_{\varepsilon}<x \;\text{and}\;X_{\tau
	(x)}>x+\varepsilon^{\varrho} \,|\, \tau (x)\leq \varepsilon) \\
	&+ P (X_{\varepsilon}<x \;\text{and}\; X_{\tau
	(x)}\leq x+\varepsilon^{\varrho} \,|\, \tau (x)\leq
	\varepsilon) \\
	&=o (1)+o (1),
\end{align*}
the last by relation~\eqref{eq:shortTimeAsymptoticsC}, which implies
that $P ( X_{\tau
	(x)}\leq x+\varepsilon^{\varrho} \,|\, \tau (x)\leq
	\varepsilon) \rightarrow 0$ 
as $\varepsilon \rightarrow 0$.

Finally, recall that $P \{ M<x \} = 1-u(x)$ is equal to the sum of the
three probabilities \eqref{eq:term-1}, \eqref{eq:term-2}, and
\eqref{eq:term-3} above. Therefore, $$\varepsilon (1-u(x)) =
\frac{1}{2} \varepsilon \bigl ( 1-u(x) \bigr )^2 + \frac{1}{2} \varepsilon +
\varepsilon \int_{-\infty}^{\infty} \bigl ( u(x)-u(x-y) \bigr ) \,\dd
\nu(y) + o(\varepsilon).$$ Dividing both sides by $\varepsilon$, then letting
$\varepsilon \rightarrow 0$, we conclude that $$\int_{-\infty}^{\infty}
\bigl ( u(x)-u(x-y) \bigr ) \,\dd \nu(y) + \frac{1}{2} (u(x))^2 = 0.$$
\end{proof}

\section{A Priori Bounds for $u(x)$}\label{sectionaprioribound}

The first step toward establishing the sharp asymptotic formula
\eqref{mainresultequation} will be to show that the function $u$
satisfies the rough asymptotic formula \eqref{eq:orderOfMagnitude}.
We will give two different arguments, one probabilistic, the other
analytic, the first showing that the \emph{particular} function $u$
defined by \eqref{eq:tailDF} satisfies the inequalities
\eqref{eq:orderOfMagnitude}, the second proving the following
(superficially) more general result. (It will follow from the
Feynman-Kac formula~\eqref{feynmankacformulaforu} below that the
solution to the boundary value problem \eqref{initialvalueproblem} is
unique, hence must coincide with \eqref{eq:tailDF}.)

\begin{prop}\label{aprioribound}
Let $u(x)$ be a continuous positive solution to the boundary value
problem \eqref{initialvalueproblem}, and suppose that $u(x)
\rightarrow 0$ as $x \rightarrow \infty$. Then there exist positive
constants $C_1$ and $C_2$ such that
\begin{equation}\label{aprioriboundforu}
\frac{C_1}{x^{\alpha/2}} \leq u(x) \leq \frac{C_2}{x^{\alpha/2}}
\end{equation}
for all  $x\geq 1$.
\end{prop}

\subsection{Probabilistic approach}\label{ssec:probabilistic} These
arguments apply specifically to the tail distribution function $u (x)$
of the maximal displacement $M$ of a CBSS process. Recall that in a
CBSS process, the number of particles alive at time $t$ is a standard
Yule (binary fission) process. The CBSS can be constructed by first
running a Yule process $Y_{t}$, then running independent symmetric
$\alpha-$stable processes along the edges of the resulting
genealogical tree. The Yule process itself can be built by first
constructing a discrete-time double-or-nothing Galton-Watson process
(i.e., a Galton-Watson process whose offspring distribution is
$p_{0}=p_{2}=1/2$) and then attaching independent unit exponential
random variables to the edges of the resulting Galton-Watson tree.

\begin{proof}
[Proof of the lower bound $u (x)\geq C_{1}/x^{\alpha /2}$.]  Denote by
$\xi$ the total progeny of the Yule process, that is, the number of
distinct particles born in the course of the branching
process. Equivalently, $\xi$ is $1+$the number of edges in the
genealogical tree. A well known (but somewhat difficult to
trace\footnote{The probability generating function of $\xi$ was
derived by I. J. Good~\cite{good} in 1949, and related results were
later obtained by Dwass~\cite{dwass} and Pakes~\cite{pakes}. It was
known to T. Harris~\cite{harris-firstPassage} that in the special case
where the offspring distribution is the geometric distribution with
mean $1$ the genealogical tree of the associated Galton-Watson process
is the excursion tree of a simple random walk excursion, from which it
follows directly that $\xi$ is $1/2$ the number of steps in the simple
random walk excursion.}) result from the elementary theory of
Galton-Watson processes has it that for a critical Galton-Watson
process whose offspring distribution has positive, finite variance,
\begin{equation}\label{eq:Galton-Watson-progeny}
	P\{\xi \geq m \}\asymp \frac{1}{\sqrt{m}}.
\end{equation}
Hence, there exists $C>0$ such that with probability at least
$C/\sqrt{m}$ the Yule tree has at least $m$ branches. 

The  branch lengths of the Yule tree are independent unit
exponentials, and so the spatial displacements $D_{e}$ of particles along
these edges $e$ are unit exponential mixtures of symmetric $\alpha-$stable
random variables $X_{t}$. Consequently, since the tail of a symmetric
$\alpha-$stable random variable is regularly varying with exponent
$\alpha$, there is a constant $C'>0$ such that, conditional on the
Galton-Watson tree, for each edge $e$
\begin{equation}\label{eq:stable-mixture-tail}
	P\{|D_{e}| \geq 3 m^{1/\alpha}\}\geq C'/m.
\end{equation}
Therefore, since the random variables $D_{e}$ are conditionally
independent given the Galton-Watson tree, it follows from
\eqref{eq:Galton-Watson-progeny} and \eqref{eq:stable-mixture-tail}
that with probability at least $C''/\sqrt{m}$ there will be some edge
$e$ of the tree for which $|D_{e}|\geq 3m^{1/\alpha}$.  But on this
event there must be at least one particle of the CBSS process that
finds its way out of the interval
$[-m^{1/\alpha},m^{1/\alpha}]$. Since the CBSS process is invariant
under reflection of the space axis, it follows that
\[
	u (m^{1/\alpha})= P\{M\geq m^{1/\alpha} \}\geq C''/\sqrt{m}.
\]
\end{proof}

\begin{proof}
[Proof of the upper bound $u (x)\leq C_{2}/x^{\alpha /2}$] This relies
on the following elementary property of the CBSS process: the mean
particle density at location $dx$ at time $t$ is $f_{t} (x)\,dx$,
where $f_{t} (x)$ is the density of the symmetric $\alpha-$stable
random variable $X_{t}$.  Consequently, for any $x\geq 0$ and $t>0$,
the conditional expectation of the number of particles to the right of
$x$ at time $t$ given that some particle of the CBSS reaches the
halfline $[x,\infty)$ before time $t$ is at least $1/2$. It follows
that
\[
	u (x)= P\{M\geq x \}\leq 2 P\{X_{t}\geq x \} +P\{Y_{t}\geq 1 \},
\]
where $Y_{t}$ is the skeletal Yule process, $M$ is the maximal
displacement of the CBSS process, and $X_{t}$ is a generic symmetric
$\alpha-$stable process. By setting $t=x^{-\alpha /2}$ and using the
fact that the distribution of $X_{1}$ has regularly varying tail with
exponent $\alpha$ and the fact (essentially Kolmogorov's theorem for
critical branching processes) that $P\{Y_{t}\geq t \}\sim C/t$, we
obtain the desired estimate
\[
	u (x)\leq C' /x^{\alpha /2}.
\]
\end{proof}

\subsection{Analytic approach}\label{ssec:analytic}

We shall prove Proposition~\ref{aprioribound} in general by first
establishing a comparison principle for the boundary value problem
\eqref{initialvalueproblem}, then comparing our $u(x)$ to a explicit
supersolutions and subsolutions of \eqref{initialvalueproblem}, both
of which decay to zero as a constant times $x^{-\alpha/2}$.  (Thanks
to Professor Luis Silvestre for suggesting this.)

\begin{prop}[Comparison Principle]\label{comparisonprinciple}
Let $u(x)$ be a continuous positive solution to the boundary value
problem \eqref{initialvalueproblem}, and suppose that 
$u(x) \rightarrow 0$ as $x \rightarrow \infty$.\\ 
(A) Suppose that $U(x)$ is a continuous positive super-solution to
\eqref{initialvalueproblem}, meaning that
\begin{equation*}
\begin{cases}
(-\Delta)^{\alpha/2} U(x) + \frac{1}{2}(U(x))^2 \geq 0 \quad & \text{for $x>0$},\\
U(x) \geq 1 & \text{for $x \leq 0$}.
\end{cases}
\end{equation*}
Furthermore, assume that $U(x) \rightarrow 0$ as $x \rightarrow
\infty$. Then, $$u(x) \leq U(x) \quad \text{for all $x \in \R$}.$$ (B)
Suppose that $V(x)$ is a continuous positive sub-solution to
\eqref{initialvalueproblem}, meaning that
\begin{equation*}
\begin{cases}
(-\Delta)^{\alpha/2} V(x) + \frac{1}{2}(V(x))^2 \leq 0 \quad & \text{for $x>0$},\\
V(x) \leq 1 & \text{for $x \leq 0$}.
\end{cases}
\end{equation*}
Furthermore, assume that $V(x) \rightarrow 0$ as $x \rightarrow \infty$. Then, $$u(x) \geq V(x) \quad \text{for all $x \in \R$}.$$
\end{prop}

\begin{proof}
We will only prove part (A). The proof of part (B) can be done in an
analogous manner.

We proceed by contradiction. Suppose that $u(x_0)>U(x_0)$ at some point
$x_0 \in \R$. Then $(u-U)(x)$, a continuous function that is
non-positive for $x \leq 0$ and goes to zero as $x \rightarrow
\infty$, would attain a strictly positive global maximum value at a
certain point $x_1>0$: $$(u-U)(x_1) = \max_{x \in \R} (u-U)(x) > 0.$$
Now consider the quantity $(-\Delta)^{\alpha/2} (u-U)(x_1)$. On one
hand, $$(-\Delta)^{\alpha/2} (u-U)(x_1) = \int_{-\infty}^{\infty}
\frac{(u-U)(x_1)-(u-U)(y)}{|x_1-y|^{1+\alpha}} \,\dd y \geq 0$$
because $(u-U)(x_1) \geq (u-U)(y)$ for all $y$. On the other hand,
\begin{equation*}
\begin{aligned}
(-\Delta)^{\alpha/2} (u-U)(x_1) & = (-\Delta)^{\alpha/2}u(x_1) - (-\Delta)^{\alpha/2}U(x_1)\\
& \leq -\frac{1}{2}(u(x_1))^2 +\frac{1}{2}(U(x_1))^2\\
& < 0.
\end{aligned}
\end{equation*}
This is a contradiction. Thus, $u(x) \leq U(x)$ for all $x \in \R$.
\end{proof}

\begin{proof}[Proof of Proposition  \ref{aprioribound}]
Consider the function
\begin{equation*}
w(x) =
\begin{cases}
(1+x)^{-\alpha/2} & \text{for $x>0$},\\
1 & \text{for $x \leq 0$}.
\end{cases}
\end{equation*}
We will show that, for a large enough constant $C$, $Cw(x+1)$ is a supersolution to \eqref{initialvalueproblem}, and for a small enough positive constant $C$, $Cw(x+1)$ is a subsolution to \eqref{initialvalueproblem}. Notice that $(-\Delta)^{\alpha/2} (Cw)(x+1) + \frac{1}{2}(Cw(x+1))^2 = C \bigl ( (-\Delta)^{\alpha/2} w(x+1) + \frac{1}{2}C(w(x+1))^2 \bigr )$. Hence it suffices to show $$\sup_{x>0} \frac{-(-\Delta)^{\alpha/2} w(x+1)}{\frac{1}{2}(w(x+1))^2} < \infty \quad \text{and} \quad \inf_{x>0} \frac{-(-\Delta)^{\alpha/2} w(x+1)}{\frac{1}{2}(w(x+1))^2} > 0.$$ Because $-(-\Delta)^{\alpha/2} w(x+1)$ is obviously continuous for $x \in [0,\infty)$, it eventually boils down to proving
\begin{equation}\label{explicitsuperandsubsolutions}
\limsup_{x \rightarrow \infty} \frac{-(-\Delta)^{\alpha/2} w(x)}{x^{-\alpha}} < \infty \quad \text{and} \quad \liminf_{x \rightarrow \infty} \frac{-(-\Delta)^{\alpha/2} w(x)}{x^{-\alpha}} > 0.
\end{equation}

Now for any $x>0$, let us compute
\begin{equation*}
\begin{aligned}
-(-\Delta)^{\alpha/2} w(x) & = -\int_{-\infty}^\infty \frac{w(x)-w(y)}{|x-y|^{1+\alpha}} \,\dd y\\
& = -\int_{-\infty}^{-1} \frac{(1+x)^{-\alpha/2}-1}{(x-y)^{1+\alpha}} \,\dd y -\int_{-1}^\infty \frac{(1+x)^{-\alpha/2}-(1+y)^{-\alpha/2}}{|x-y|^{1+\alpha}} \,\dd y\\
& := -A-B
\end{aligned}
\end{equation*}
The first integral can be easily evaluated: $$-A = \frac{1}{\alpha} \cdot (1-(1+x)^{-\alpha/2}) \cdot (x+1)^{-\alpha} \sim \frac{1}{\alpha} \cdot x^{-\alpha} \quad \text{as $x \rightarrow \infty$}.$$ To deal with the second integral $B$, consider an auxiliary function
\begin{equation*}
\begin{aligned}
F(x) & := \int_0^\infty \frac{x^{-\alpha/2}-y^{-\alpha/2}}{|x-y|^{1+\alpha}} \,\dd y\\
& = \int_0^\infty \frac{x^{-\alpha/2}-\lambda^{-\alpha/2}z^{-\alpha/2}}{|x-\lambda z|^{1+\alpha}} \,\lambda\dd z \qquad \text{($y=\lambda z$)}
\end{aligned}
\end{equation*}
where $\lambda>0$ is an arbitrarily chosen constant. Then, observe that
\begin{equation*}
\begin{aligned}
F(\lambda x) & = \int_0^\infty \frac{\lambda^{-\alpha/2}x^{-\alpha/2}-\lambda^{-\alpha/2}z^{-\alpha/2}}{|\lambda x-\lambda z|^{1+\alpha}} \,\lambda\dd z\\
& = \lambda^{-3\alpha/2} \int_0^\infty \frac{x^{-\alpha/2}-z^{-\alpha/2}}{|x-z|^{1+\alpha}} \,\dd z\\
& = \lambda^{-3\alpha/2} F(x)
\end{aligned}
\end{equation*}
for all $\lambda>0$ and all $x>0$. This scaling property of $F$ immediately implies that there exists constant $C$ such that $$F(x)=C \cdot x^{-3\alpha/2}.$$ To relate $F(x)$ to our integral $B$, we notice that
\begin{equation*}
\begin{aligned}
F(1+x) & = \int_0^\infty \frac{(1+x)^{-\alpha/2}-y^{-\alpha/2}}{|(1+x)-y|^{1+\alpha}} \,\dd y \qquad \text{(by definition of $F$)}\\
& = \int_{-1}^\infty \frac{(1+x)^{-\alpha/2}-(1+z)^{-\alpha/2}}{|x-z|^{1+\alpha}} \,\dd z \qquad \text{($y=1+z$)}\\
& = B.
\end{aligned}
\end{equation*}
Hence $$B = F(1+x) \sim C \cdot x^{-3\alpha/2} = o(x^{-\alpha}).$$ Thus, $$-(-\Delta)^{\alpha/2} w(x) = -A-B = \frac{1}{\alpha} \cdot x^{-\alpha} + o(x^{-\alpha}),$$ verifying \eqref{explicitsuperandsubsolutions}.

Therefore, there exist positive constants $C'_1$ and $C'_2$ such that $C'_1w(x)$ is a subsolution to \eqref{initialvalueproblem} and $C'_2w(x)$ is a supersolution to \eqref{initialvalueproblem}. By Proposition \ref{comparisonprinciple}, $C'_1w(x) \leq u(x) \leq C'_2w(x)$ for all sufficiently large $x$. Since $w(x) \sim x^{-\alpha/2}$ as $x \rightarrow \infty$, there are positive constants $C_1$ and $C_2$ such that $C_1 x^{-\alpha/2} \leq u(x) \leq C_2 x^{-\alpha/2}$ for all sufficiently large $x$, proving \eqref{aprioriboundforu}.
\end{proof}

\section{Proof of Theorem \ref{mainresult}}\label{sectionfeynmankac}

\subsection{ Feynman-Kac Representation of Solutions} Our approach to
Theorem~\ref{mainresult} will rely on an analogue of the Feynman-Kac
formula for solutions to pseudo-differential equations of the form
$-(-\Delta)^{\alpha/2} v(x) = q(x)v(x)$.  The operator
$-(-\Delta)^{\alpha/2}$ is the infinitesimal generator of the
symmetric $\alpha$-stable process, and hence the Feynman-Kac
representations will be functional integrals with respect to paths
$X_{t}$ of the symmetric $\alpha$-stable process. Denote by $P^{x}$
and $E^{x}$ the probability and expectation operators under which the
initial point of the process is $X_{0}=x$, and recall that
$\tau_{0}^{-}$ is the first-passage time to the half-line
$(-\infty,0)$.

\begin{thm}[Feynman-Kac Formula]\label{feynmankac}
Let $v:\zz{R}\rightarrow \zz{R}$ be a bounded, continuous solution of
\begin{equation}\label{feynmankacequation}
-(-\Delta)^{\alpha/2} v(x) = q(x)v(x) \quad \text{for all $x>0$},
\end{equation}
where $q(x)$ is a nonnegative and continuous. Then $$Z_t = \exp
\bigl ( -\int_0^{t \wedge \tau^{-}_{0}} q(X_s) \,\dd s \bigr ) \cdot v(X_{t
\wedge \tau_0})$$ is a bounded martingale with respect to the
filtration $\{ \mathscr{F}_{t \wedge \tau^{-}_{0}}^x \}_{t \geq
0}$. Consequently, by the Optional Stopping Theorem, for any stopping
time $\tau \leq \tau^{-}_{0}$,
\begin{equation}\label{feynmankacformula}
v(x) = E^x[\exp \bigl ( -\int_0^\tau q(X_s) \,\dd s \bigr ) \cdot v(X_\tau)].
\end{equation}
\end{thm}

The Feynman-Kac formula has been proved to hold  for arbitrary Markov
processes satisfying the Feller property -- see, for instance,
\cite{rogers-williams}, and \cite{kesten} for some of the history of
the formula. Theorem\ref{feynmankac}  is a special case, as
the symmetric $\alpha-$stable process is Feller. 

\begin{cor}
If $u(x)$ is a solution to the boundary value problem
\eqref{initialvalueproblem}, then for any stopping time $\tau \leq
\tau^{-}_{0}$,
\begin{equation}\label{feynmankacformulawithtau}
u(x) = E^x[\exp \bigl ( -\frac{1}{2} \int_0^\tau u(X_s) \,\dd s \bigr ) \cdot u(X_\tau)].
\end{equation}
In particular,
\begin{equation}\label{feynmankacformulaforu}
u(x) = E^x[\exp \bigl ( -\frac{1}{2} \int_0^{\tau^{-}_{0}} u(X_s) \,\dd s \bigr )].
\end{equation}
\end{cor}

\begin{proof}
The representation \eqref{feynmankacformulawithtau} directly follows
from Theorem \ref{feynmankac} by setting $q(x) = \frac{1}{2}
u(x)$. The result \eqref{feynmankacformulaforu} follows from setting
$\tau=\tau^{-} (0)$, since $u(X_{\tau_0}) = 1$.
\end{proof}

\subsection{Consequences of the Feynman-Kac
formula}\label{ssec:consequencesFK}

Formula \eqref{feynmankacformulaforu} restricts the decay of $u(x)$
both above and below, because the function $u$ appears on both sides
of \eqref{feynmankacformulaforu} but with opposite signs. When
combined with the \emph{a priori} estimates of
Proposition~\ref{aprioribound}, the integral representation
\eqref{feynmankacformulaforu} will lead to sharp asymptotic estimates,
as we will show in section~\ref{mainresult}. In this section we
collect some preliminary consequences of the representation
\eqref{feynmankacformulaforu}. Henceforth, we will use the notational
shorthand 
\begin{equation}\label{eq:y-t}
	\Psi_{t}= \int_{0}^{t} u(X_{s})
	\,\dd s
\end{equation}
for the path integrals that occur in the Feynman-Kac formulas. 
Recall that for any interval $J$ the random variable $\sigma_{J}$ is
the time of first exit from $J$, and $\nu_{J}$ is the time of the
first jump of size $X_{t}-X_{t-}\in J$.

\begin{prop}\label{proposition:oneBigJump1}
Fix $\delta \in (0,\frac{1}{2})$ and abbreviate $\sigma
=\sigma_{(x-\delta x,x+\delta x)}$. For all sufficiently small
$\delta$, as $x \rightarrow \infty$,
\begin{equation}\label{eq:oneBigJump1}
	E^{x} \exp \left\{-\Psi_{\sigma}/2\right\} u (X_{\sigma
	})\mathds{1} \{X_{\sigma}\geq \delta x \}=o (u (x)),
\end{equation}
and consequently,
\begin{equation}\label{eq:BigJumpCor}
	E^{x} \exp \left\{-\Psi_{\sigma}/2\right\}\mathds{1}
	\{X_{\sigma}\geq \delta x \}=o (1). 
\end{equation}
\end{prop}

\begin{proof}
The monotonicity of $u$ and the \emph{a priori} bounds
\eqref{aprioriboundforu} imply that the ratio $u (X_{\sigma})/u (x)$
remains bounded above on the event $X_{\sigma}\geq \delta x$ by a
constant $C=C_{\delta}<\infty$ depending on $\delta >0$ but not on
$x$. Hence, the second relation \eqref{eq:BigJumpCor} will follow from
the first relation \eqref{eq:oneBigJump1}. Now consider the
exponential $\exp \{-\Psi_{\sigma}/2 \}$.  The integrand $u (X_{s})$ is
bounded below by $u (x+\delta x)$ up to time $\sigma$, by the
monotonicity of $u$, and so by the \emph{a priori} bounds
\eqref{aprioriboundforu}, with $C=C_{\delta}$ as above,
\[
	\Psi_{\sigma}\geq C\sigma x^{-\alpha /2}.
\]
Hence, on the event $\sigma >x^{\alpha /2+\eta}$ the exponential
$e^{-\Psi_{\sigma}/2}$ will be bounded above by  $\exp \{-Cx^{\eta} \}=o
(u (x))$. On the other hand, the scaling law \eqref{scalingproperty}
implies that the distribution of $\sigma /x^{\alpha}$ under $P^{x}$ is
the same as that of $\sigma$ under $P^{1}$, so as $x \rightarrow
\infty$ the probability that $\sigma \leq x^{\alpha /2+\eta}$
converges to zero, for any $\eta <\alpha /2$. Thus,
\begin{align*}
	E^{x}e^{-\Psi_{\sigma}/2} u (X_{\sigma })\mathds{1}
	\{X_{\sigma}\geq \delta x \} &\leq Cu (x)
	E^{x}e^{-\Psi_{\sigma}/2} \\
	&\leq  Cu (x) E^{x}e^{-\Psi_{\sigma}/2} (\mathds{1}\{\sigma \geq
	x^{\alpha /2+\eta} \}+\mathds{1}\{\sigma < x^{\alpha
	/2+\eta} \} )\\
	&=Cu (x) (o (1)+o (1)).
\end{align*}
\end{proof}

Proposition \ref{proposition:oneBigJump1} implies that for large $x$
the expectation in the Feynman-Kac formula
\eqref{feynmankacformulawithtau}, with $\tau =\sigma$, is dominated by
those sample paths that exit the interval $(x-\delta x,x+\delta x)$ by
jumping to the interval $(-\infty , \delta x)$. The next result
asserts that the relative contribution from those paths for which the
jump lands somewhere in $(-\delta x,\delta x)$ is small.

\begin{prop}\label{proposition:oneBigJump2}
For each $\varepsilon>0$ there exists $\delta >0$ such that 
for all sufficiently large $x$
\begin{equation}\label{eq:oneBigJump2}
	E^{x}\exp \left\{ -\Psi_{\sigma}/2 \right\} \mathds{1} \{X_{\sigma}\in [-2\delta
	x,2\delta x] \}\leq \varepsilon u (x),
\end{equation}
where $\sigma =\sigma_{(x-\delta x,x+\delta x)}$.
\end{prop}

\begin{proof}
If $\delta <1/4$ then the event $X_{\sigma}\in [-2\delta x,2\delta x]$
can only occur if there is a jump of size $\Delta =X_{\sigma}-X_{\sigma -}<-x+3\delta x$ at time
$\sigma$. Moreover, because $X_{\sigma -}\in [x-\delta
\mathcal{X}+\delta x]$, the jump must be the \emph{first} jump of
magnitude more than $2\delta x$, and so $\sigma =\nu $, where $\nu
=\nu_{(-\infty ,-x+3\delta x)}$.  In order that
$X_{\sigma}\in [-2\delta x,2\delta x]$, 
the size of the jump
must satisfy 
\[
	\Delta \in [-x-4\delta x,-x+4\delta x].
\]
Similarly, if at time $\sigma$ the process $X_{s}$ makes a jump of
size $\Delta <-x-\delta x$ then $X_{\sigma}<0$ and so $\sigma
=\tau^{-} (0)$.

By Lemma~\ref{lemma:jumpIndependence}, the random variable $\Delta$ is
independent of the $\sigma -$algebra $\mathcal{F}_{\nu -}$ under
$P^{x}$, and furthermore the distribution of $\Delta$ is
\[
	P^{x}\{\Delta \leq -x-tx \}=\frac{\lambda (-\infty
	,-x-tx)}{\lambda (-\infty ,-x+2\delta x)} = \left(
	\frac{1-3\delta }{1+t }\right)^{\alpha} .
\]
Hence, since
$\tau^{-} (0)=\nu$ on the event $\{\sigma =\nu  \}\cap \{\Delta\leq
-x-\delta x \}$, the Feynman-Kac formula \eqref{feynmankacformulaforu}
implies that
\begin{align*}
	u (x)&\geq E^{x} (\exp \{-\Psi_{\sigma}/2 \} \mathds{1}\{\sigma =\nu  \}
	\mathds{1}\{\Delta\leq -x-\delta x \} )  \\
	&=E^{x} ( \exp \{-\Psi_{\sigma}/2 \} \mathds{1}\{\sigma =\nu  \})
	P^{x}\{\Delta\leq -x-\delta x \}  .
\end{align*}
But the independence of $\Delta$ and $\mathcal{F}_{\nu -}$ also implies that 
\begin{align*}
	&E^{x} (\exp \{-\Psi_{\sigma}/2 \} \mathds{1}\{\sigma =\nu  \}
	\mathds{1}\{\Delta\in [-x-4\delta x,-x+4\delta x] \} )\\
	=&E^{x} ( \exp \{-\Psi_{\sigma}/2 \} \mathds{1}\{\sigma =\nu  \}) 
	P^{x}\{\Delta\in [-x-4\delta x,-x+4\delta x]\}  ,
\end{align*}
since $\Psi_{\nu}$ is measurable with respect to $\mathcal{F}_{\nu -}$,
so it now follows that
\begin{align*}
	E^{x} (\exp \{-&\Psi_{\sigma}/2 \} \mathds{1} \{X_{\sigma}\in [-2\delta
	x,2\delta x] \;\text{and}\;\sigma  =\nu \}\\
	&\leq  u (x) \cdot\frac{P^{x}\{\Delta\in [-x-4\delta
	x,-x+4\delta x]\}}{P^{x}\{\Delta\leq -x-\delta x \}} .
\end{align*}
The ratio of the two probabilities on the right is $O (\delta)$
(uniformly in $x$, by scaling), so this proves \eqref{eq:oneBigJump2}.
\end{proof}

\begin{prop}\label{uniformratio}
For any $\varepsilon>0$, there exists $\delta>0$ such that for all $x\geq 1$,
\begin{equation}\label{eq:continuity-uniform}
1-\varepsilon
\leq \frac{u((1+\delta)x)}{u(x)} <1.
\end{equation}
\end{prop}

\begin{proof}
The law of the symmetric $\alpha-$stable process $X_{s}$ under
$P^{x+\delta x}$ is the same as that of $X_{s}+\delta x$ under
$P^{x}$, and the first passage time $\tau^{-} (\delta x)$ under
$P^{x+\delta x}$ is the same as that of $\tau^{-} (0)$ under $P^{x}$. Hence,
by the Feynman-Kac
formula \eqref{feynmankacformulawithtau},
\begin{equation*}
\begin{aligned}
\frac{u((1+\delta)x)}{u(x)} & = \frac{E^{(1+\delta)x}[\exp \{-\Psi_{\tau^{-} (\delta x)}/2 \}  \cdot u(X_{\tau^{-} (\delta x)})]}{u(x)}\\ 
& = \frac{E^{x}[\exp \bigl ( -\frac{1}{2} \int_0^{\tau^{-}(0)} u(X_s+\delta x) \,\dd s \bigr ) \cdot u(X_{\tau^{-}(0)}+\delta x)]}{u(x)}\\
& \geq \frac{E^{x}[\exp \bigl ( -\frac{1}{2} \int_0^{\tau^{-}(0)} u(X_s) \,\dd s \bigr ) \cdot \mathds{1} \{ X_{\tau^{-}(0)} \leq -\delta x \}]}{E^{x}[\exp \bigl ( -\frac{1}{2} \int_0^{\tau^{-}(0)} u(X_s) \,\dd s \bigr )]}.
\end{aligned}
\end{equation*}
Thus, to prove Proposition \ref{uniformratio} it suffices to show
that for any $\varepsilon>0$ there exists $\delta>0$ such that
\begin{equation*}
	 E^{x}[\exp \{-\Psi_{\tau^{-} (0)}/2 \}\cdot \mathds{1} \{
	 X_{\tau^{-}(0)} \in (-\delta x,0] \}]\leq \;  \varepsilon
	 u(x)  
\end{equation*}
for all sufficiently large  $x$. But this follows directly from
Proposition~\ref{proposition:oneBigJump2}.
\end{proof}

Propositions
\ref{proposition:oneBigJump1}--\ref{proposition:oneBigJump2} imply
that the expectation in the Feynman-Kac formula
\eqref{feynmankacformulaforu} is dominated by paths $X_{s}$ for which
the first escape from $[x-\delta x,x+\delta x]$ coincides with the
first jump of size $<-x+2\delta x$. On this event, the stopping time
$\tau^{-}_{0}$ coincides with the time $\sigma$ of first exit from
$[x-\delta x,x+\delta x]$ and with the time $\nu$ of the first jump of
size $<-x+2\delta x$. The importance of Proposition~\ref{uniformratio}
is that it guarantees that on this event the value of $\Psi_{\tau^{-}
(0)}$ is nearly the same as $\nu u (x)$. Thus, it is not unreasonable
to hope that the Feynman-Kac
expectation \eqref{feynmankacformulaforu} should be well-approximated
by $E^{x}\exp \{-\nu u (x)/2 \}$. Since $\nu$ is the first occurrence
time in a Poisson process of rate 
$\lambda [x-2\delta x,\infty )=Cx^{-\alpha}$, where $C=
(1-2\delta)^{-\alpha}/\alpha$, it is exponentially
distributed and so the latter expectation can be evaluated exactly:
\begin{equation}\label{eq:exactEval}
		E^{x}e^{-\nu u (x) /2} = C/ (C+x^{\alpha}u (x)).
\end{equation}
Given this, Theorem~\ref{mainresult}  will follow, because together
with the Feynman-Kac formula it leads to  the limiting 
relation 
\[
	u (x)\sim C/ (C+x^{\alpha}u (x)),
\]
from which \eqref{mainresultequation} can be easily deduced. 

To justify the replacement of the Feynman-Kac expectation
\eqref{feynmankacformulaforu} by the expectation $E^{x}\exp \{-\nu u
(x)/2 \}$, we must verify that the contribution to this last
expectation from paths for which $\tau^{-} (0)=\sigma =\nu$ does not
hold is small.

\begin{prop}\label{proposition:nu-est}
Fix $\delta >0$ and let $\sigma $ be
the time of first exit from $(x-\delta x,x+\delta x)$ and $\nu $ the
time of the first jump of size $<-x+2\delta x$. 
For any $\varepsilon >0$, if $\delta >0$ is sufficiently small then
as $x \rightarrow \infty$,
\begin{equation}\label{eq:nu-est}
	E^{x}\exp \{- (1+\varepsilon )\nu u (x) /2\} \mathds{1}\{\sigma \not =\nu  \} \leq
	o ( u (x)). 
\end{equation}
\end{prop}

\begin{proof}
The event $\{\sigma \not =\nu \}$ can occur only if $\sigma <\nu$ and
$X_{\sigma}\geq \delta x$. Moreover, by the strong Markov property for
the underlying Poisson point process, the conditional distribution of
the residual waiting time $\nu -\sigma$ given $\mathcal{F}_{\sigma}$
on the event $\{\sigma \not =\nu \}$ is the same as the
\emph{unconditional} distribution of $\nu$, and so 
\begin{align*}
	E^{x}e^{- (1+\varepsilon )\nu u (x) /2}  \mathds{1} \{\sigma \not =\nu \} &\leq
	E^{x}e^{- (1+\varepsilon ) (\nu -\sigma) u (x) /2} e^{-
	(1+\varepsilon )\sigma u (x)/2} \mathds {1} \{\sigma \not =\nu \} \\
	&=E^{x}e^{-(1+\varepsilon )\nu u (x) /2} E^{x}
	e^{-(1+\varepsilon )\sigma u (x)/2}\mathds{1} \{\sigma \not =\nu
	\} \\
	&\leq E^{x}e^{-(1+\varepsilon )\nu u (x) /2} E^{x}
	e^{-\Psi_{\sigma}/2} \mathds{1} \{\sigma \not =\nu\} \\
	&\leq   E^{x}e^{-(1+\varepsilon )\nu u (x) /2} E^{x}
	e^{-\Psi_{\sigma}/2} \mathds{1} \{X_{\sigma } \geq \delta x \}. 
\end{align*}
(The third inequality holds by Proposition~\ref{uniformratio},
provided $\delta >0$ is sufficiently small and $x\geq 1$.)
Hence, by Proposition~\ref{proposition:oneBigJump1}, as $x \rightarrow
\infty$,
\begin{align*}
	E^{x}e^{- (1+\varepsilon )\nu u (x) /2}  \mathds{1} \{\sigma
	\not =\nu \} 
	&\leq o ( E^{x}e^{-(1+\varepsilon )\nu u (x) /2} ).
\end{align*}
Thus, to complete the proof we need only show that 
\[
	E^{x}e^{-(1+\varepsilon )\nu u (x) /2} =O (u (x)).
\]
But this follows routinely from the fact that $\nu $ is exponentially
distributed with rate $\lambda [x-2\delta x,\infty )=Cx^{-\alpha}$,
where $C= (1-2\delta)^{-\alpha}/\alpha$:
\[
	E^{x}e^{-(1+\varepsilon )\nu u (x) /2} = C/ (C+
	(1+\varepsilon) x^{\alpha}u (x)).
\]
The \emph{a priori} estimates \eqref{aprioriboundforu} now yield the
desired conclusion.
\end{proof}

\subsection{Proof of Theorem \ref{mainresult}}
Fix $\delta >0$ small and write $\sigma =\sigma_{(x-\delta x,x+\delta
x)}$ for the first exit time from the interval $(x-\delta x,x+\delta
x)$ and $\nu =\nu_{(-\infty ,-x+2\delta x)}$ for the time of the first
discontinuity of size $<-x+2\delta x$. By
Propositions~\ref{proposition:oneBigJump1}--\ref{proposition:oneBigJump2},
for any $\varepsilon >0$ there exists $\delta >0$ so small that
\begin{align*}
	u(x) &= E^x \exp \left\{ -\Psi_{\tau^{-} (0)}/2 \right\} \\
\notag 	 &\leq  (1-\varepsilon )^{-1} E^x \exp \left\{
-\Psi_{\sigma}/2\right\} \mathds{1} \{ \tau^{-} (0) =\sigma =\nu \}.  
\end{align*}
On the event $\{ \sigma=\nu =\tau^{-} (0) \}$, the path $X_{s}$ remains in
the interval $(x-\delta x,x+\delta x)$ up to time $\tau^{-} (0)$, so
by Proposition~\ref{uniformratio}
the path integral in the exponential is approximately $\nu  u (x)$:
more precisely, for any $\varepsilon >0$ there exists $\delta >0$ so
small that 
\begin{align}\label{eq:bounds}
	u (x) &\leq (1+\varepsilon ) E^{x} \exp \left\{- (1-\varepsilon )\nu  u (x)/2
	\right\} \mathds{1}  \{ \tau^{-} (0)=\sigma =\nu  \} \\
\notag 	&\leq (1+\varepsilon ) E^{x} \exp \left\{- (1-\varepsilon )\nu  u (x)/2
	\right\} \quad \text{and}\\
\notag 	u (x) &\geq (1-\varepsilon ) E^{x} \exp \left\{-(1+\varepsilon
	)\nu u (x)/2 \right\} \mathds{1} \{\tau^{-}(0)=\sigma =\nu
	\}. 
\end{align}
The expectation in the upper bound can be evaluated exactly, as in the
proof of Proposition~\ref{proposition:nu-est}, using the fact that
$\nu$ is exponentially distributed. This yields the inequality 
\begin{equation}\label{eq:upper}
	u (x)\leq \frac{(1+\varepsilon)x^{-\alpha}
	(1-2\delta)^{-\alpha}  /\alpha }{(1-\varepsilon)u 
	(x)/2+x^{-\alpha} (1-2\delta)^{-\alpha}/\alpha} .
\end{equation}

To obtain a usable lower bound from the last inequality in
\eqref{eq:bounds} we use Proposition~\ref{proposition:nu-est}. The
complement of the event $\{\tau^{-}(0)=\sigma =\nu \}$ is contained in
the union of $\{ \sigma \not =\nu \}$ with $\{X_{\sigma}\in (-2\delta
x,2\delta x) \}$. Proposition~\ref{proposition:nu-est} implies that
as $x \rightarrow   \infty$,
\[
	E^{x} \exp \left\{-(1+\varepsilon
	)\nu u (x)/2 \right\} \mathds{1} \{\sigma \not =\nu \}=o (u (x)),
\]
while Propositions~\ref{uniformratio} and
\ref{proposition:oneBigJump2} imply that for sufficiently small
$\delta >0$, if $x$ is large then
\begin{align*}
	&E^{x} \exp \left\{-(1+\varepsilon
	)\nu u (x)/2 \right\} \mathds{1} \{X_{\sigma}\in (-2\delta
	x,2\delta x) \} \\
 & \leq E ^{x} \exp \left\{-\Psi_{\sigma}/2
	\right\} \mathds{1} \{X_{\sigma}\in (-2\delta
	x,2\delta x) \} \\
&\leq\varepsilon u (x).
\end{align*}
Consequently,  for sufficiently small
$\delta >0$ and large $x$,
\begin{align*}
	&E^{x} \exp \left\{-(1+\varepsilon
	)\nu u (x)/2 \right\} \mathds{1} \{\tau^{-}(0)=\sigma =\nu
	\}\\
&\geq E^{x} \exp \left\{-(1+\varepsilon
	)\nu u (x)/2 \right\} -2\varepsilon u (x).
\end{align*}
The last expectation can now be evaluated, using once again the fact
that $\nu$ is exponentially distributed; this gives the lower bound
\begin{equation}\label{eq:lower}
	u (x)\geq \left( \frac{1}{1-2\varepsilon}\right)
	\frac{x^{-\alpha}
	(1-2\delta)^{-\alpha}  /\alpha }{(1+\varepsilon)u 
	(x)/2+x^{-\alpha} (1-2\delta)^{-\alpha}/\alpha}
\end{equation}

Since $\varepsilon >0$ and $\delta >0$ can be made arbitrarily small,
it now follows from the inequalities
\eqref{eq:upper}--\eqref{eq:lower} that 
\[
	\lim_{x \rightarrow \infty} x^\alpha (u(x))^2
	=\frac{2}{\alpha} .
\]
 Therefore, $$u(x) \sim \sqrt{\frac{2}{\alpha}}
\frac{1}{x^{\alpha /2}}.$$
\qed

\bigskip \noindent 
\textbf{Acknowledgment.} Thanks to Renming Song for pointing out a
number of minor errors in the original version.
 


\bibliographystyle{plain}
\bibliography{main}

\end{document}